\documentclass{article}

\usepackage[utf8]{inputenc}
\usepackage{amssymb, amsmath}
\usepackage{amsthm}
\usepackage{amsfonts} 
\usepackage{tikz-cd}
\usepackage{booktabs}
\usetikzlibrary{arrows}

\newtheorem{thm}{Theorem}[section]
\newtheorem{lem}{Lemma}[section]
\newtheorem{prop}{Proposition}[section]
\newtheorem{cor}{Corollary}[section]
\newtheorem{dfn}{Definition}[section]
\newtheorem{cjc}{Conjecture}[section]

\usepackage[parfill]{parskip}

\title{A refinement of the Kac polynomials for quivers with enough loops}
\author{Jiuzhao Hua} 

\begin{document}
\date{\vspace{-0.5cm}}\date{} 
\maketitle
\begin{abstract}
A conjecture of Kac now a theorem asserts that the polynomial now known as the Kac polynomial, which counts the isomorphism classes of absolutely indecomposable 
representations of a quiver over a finite field with a given dimension vector, has non-negative integer coefficients only.
In this paper, we show that, for quivers with enough loops, every Kac polynomial can be expressed as a sum of the refined Kac polynomials 
which are parametrized by tuples of partitions and have non-negative integer coefficients only. A closed formula for the refined Kac polynomials is given. We further introduce a new class of representations called blocks
and make a conjectural interpretation of the refined Kac polynomials for quivers with enough loops in terms of the numbers of block representations.
\end{abstract}

\section{Introduction}
Let $\mathbb{N}$ be the set of all non-negative integers, $\mathbb{Z}$ the ring of integers, $\mathbb{Q}$ the field of rational numbers,
and $\mathbb{F}_q$ a finite field with $q$ elements where $q$ is a prime power. Let $\Gamma = (\Gamma_0, \Gamma_1)$ be a finite quiver, 
$\alpha\in\mathbb{N}^n$ a dimension vector where $n=|\Gamma_0|$,
and $A_\Gamma(\alpha,q)$ the number of isomorphism classes of absolutely indecomposable representations of $\Gamma$ over $\mathbb{F}_q$ with dimension vector $\alpha$. 

In 1980, Kac \cite {VK 1980} proved that if $\Gamma$ has no loops then $A_\Gamma(\alpha,q)$ is non-zero if and only if $\alpha$ is
a positive root of the Kac-Moody algebra associated to $\Gamma$. In 1982, Kac \cite {VK 1983} proved that $A_\Gamma(\alpha,q)$ is always a polynomial in $q$ with integer coefficients for any quiver. 
Furthermore, Kac \cite {VK 1983} made the following two conjectures.

\begin{cjc}(Kac)
All coefficients of the polynomial $A_\Gamma(\alpha,q)$ are non-negative integers for all $\alpha\in\mathbb{N}^{n}$.
\end{cjc}

As such, $A_\Gamma(\alpha,q)$ are known as the Kac polynomials in the literature. This conjecture was proved by Crawley-Boevey and Van den Bergh \cite{C-V 2004} for indivisible dimension vectors, 
and by Mozgovoy \cite {SM 2013} for quivers with enough loops,
and by Hausel, Letellier \& Rodriguez-Villegas \cite{HLR 2013} for general quivers. A quiver has enough loops if there is at least one loop at each vertex. 

\begin{cjc}(Kac)
If $\,\Gamma$ has no loops, then the constant term of the polynomial $A_\Gamma(\alpha,q)$ is equal to the multiplicity of the root vector $\alpha$ in the Kac-Moody algebra associated to $\Gamma$.
\end{cjc}

This conjecture was again proved by Crawley-Boevey and Van den Bergh \cite{C-V 2004} for indivisible dimension vectors, 
and by Hausel \cite{TH 2010} for arbitrary dimension vectors.

Rodriguez-Villegas \cite{FRV 2011} introduced a refinement of Kac polynomials over partitions for the $g$-loop quiver and conjectured that those refined Kac polynomials have non-negative integer coefficients only.
Rodriguez-Villegas \cite{FRV 2011} also proved a formula for the values of the refined Kac polynomials at $q=1$.
In this paper, we generalize Rodriguez-Villegas' refinement by defining the refined Kac functions for an arbitrary quiver and show that they are polynomials in $q$ with non-negative integer coefficients if the quiver has
enough loops.
The main tool we use in this paper is a formula for $A_\Gamma(\alpha,q)$ by Hua \cite{JH 2000}, which is recalled here.

Without loss of generality, we fix an order on the vertices of $\Gamma$, say $\Gamma_0 = \{v_1,v_2,\cdots, v_n\}$ where $n=|\Gamma_0|$, and assume that the number of arrows from vertex $v_i$ to vertex $v_j$ is $a_{ij}\in\mathbb{N}$ for $1\le i, j \le n$. The matrix
$C(\Gamma) := [a_{ij}]_{1\le i,j \le n}$ is called the \textit{companion matrix} of $\Gamma$.
For $\alpha=(\alpha_1,\cdots,\alpha_n)\in\mathbb{N}^n$, let $\mathrm{R}_\Gamma(\alpha, \mathbb{F}_q)$ be the set of all representations of $\Gamma$ over $\mathbb{F}_q$ with dimension vector $\alpha$. $\mathrm{R}_\Gamma(\alpha, \mathbb{F}_q)$ can be identified with
the affine variety $\prod_{i,j=1}^n\mathbb{F}_q^{a_{ij}\alpha_i\alpha_j}$ and hence we have  
\begin{equation}\label{num of rep}
|\mathrm{R}_\Gamma(\alpha, \mathbb{F}_q)|=q^{\sum_{1\le i ,j \le n}a_{ij}\alpha_i\alpha_j} = q^{\alpha C(\Gamma) \alpha^t},
\end{equation}
where $\alpha^t$ is the transpose of vector $\alpha$ . 
Let $\mathrm{GL}(\alpha, \mathbb{F}_q) = \prod_{i=1}^n\mathrm{GL}(\alpha_i, \mathbb{F}_q)$
where $\mathrm{GL}(\alpha_i, \mathbb{F}_q)$ is the general linear group over $\mathbb{F}_q$ of order $\alpha_i$. 
Since for $m\in\mathbb{N}$, 
\begin{equation}\label{order of GL}
|\mathrm{GL}(m, \mathbb{F}_q)|  =  \prod_{i=0}^{m-1}(q^m - q^i) ,
\end{equation}
we have 
$$
|\mathrm{GL}(\alpha, \mathbb{F}_q)| = \prod_{i=1}^n|\mathrm{GL}(\alpha_i, \mathbb{F}_q)| = \prod_{i=1}^n\prod_{s=0}^{\alpha_i-1}(q^{\alpha_i} - q^s).
$$
In what follows, we treat $|\mathrm{R}_\Gamma(\alpha, \mathbb{F}_q)|$ and $|\mathrm{GL}(\alpha, \mathbb{F}_q)|$ as polynomial functions in $q$.
Recall that the Euler form associated to $\Gamma$ is 
\begin{equation}\label{def euler form}
\langle \alpha, \beta \rangle := \sum_{i=1}^n\alpha_i\beta_i - \sum_{i,j=1}^na_{ij}\alpha_i\beta_j \text{ for } \alpha, \beta \in \mathbb{Z}^n.
\end{equation}
Thus we have $\langle \alpha, \beta \rangle = \alpha (I - C(\Gamma)) \beta^t$ where $I$ is the identity matrix.

Let $\mathbb{Q}(q)[[X_1,\dots,X_n]]$ be the formal power series ring in $n$ variables over the rational function field $\mathbb{Q}(q)$. For a dimension vector $\alpha=(\alpha_1,\cdots,\alpha_n)\in\mathbb{N}^n$, 
we write $X^\alpha = \prod_{i=1}^nX_i^{\alpha_i}$. We have the following theorem which is equivalent to Theorem 4.6 of Hua \cite{JH 2000} (see Appendix I for the proof).

\begin{thm}\label{hua thm}
With the notation above, let $P_\Gamma(q, X_1,\cdots,X_n) $  be the formal power series in the ring $\mathbb{Q}(q)[[X_1,\dots,X_n]]$ defined as follows:
\begin{equation}\label{def PX}
P_\Gamma(q, X_1,\cdots,X_n) := 1 + \sum_{\alpha_*} \!
\Big(\prod_{k\ge 1} \frac{q^{\langle \alpha^k, \alpha^k\rangle}} {q^{\langle \beta^k, \beta^k \rangle}}  \frac{|\mathrm{R}_\Gamma(\alpha^k, \mathbb{F}_q)|}{|\mathrm{GL}(\alpha^k, \mathbb{F}_q)|} X^{k\alpha^k}\Big),
\end{equation}
where the sum runs over all tuples of dimension vectors  $\alpha_*=(\alpha^1,\alpha^2,\cdots, \alpha^r)$ such that $r\ge 1$, $\alpha^k\in\mathbb{N}^n$ for $1\le k \le r$ with $\alpha^r\ne 0$, 
and $\beta^k=\sum_{i\ge k} \alpha^i$ for $1\le k\le r$. Then we have the following identity:
\begin{equation}\label{hua formula}
P_\Gamma(q,X_1,\cdots,X_n) = \,\,\mathrm{Exp}\Big(\frac{1}{q-1} \sum_{\alpha\in\mathbb{N}^n\backslash\{0\}}\!A_\Gamma(\alpha,q)X^\alpha\Big),
\end{equation}
where the $\mathrm{Exp}$ is the plethystic exponential map (see Mozgovoy \cite{SM 2007} for its definition and properties).
\end{thm}

This paper is organized as follows. In Section 2, we define the refined Kac functions over tuples of partitions for an arbitrary quiver by refining the formal power series $P_\Gamma(q, X_1,\cdots,X_n)$.
In Section 3, we first show that, for a quiver with enough loops, if all parts of the partitions in an $n$-tuple are less than or equal to $1$, then the corresponding refined Kac function is a polynomial in $q$
with non-negative integer coefficients. Then we
construct a larger quiver $\Gamma_{\!m}$ for $m\ge 1$ when $\Gamma$ has enough loops, and then establish a link between those refined Kac functions of
$\Gamma$ parametrized by tuples of partitions with parts less than or equal to $m$ and those refined Kac polynomials of
$\Gamma_{\!m}$ parametrized by tuples of partitions with parts less than or equal to $1$. 
As a consequence of this link, every refined Kac function of $\Gamma$ must be a polynomial in $q$
with non-negative integer coefficients if $\Gamma$ has enough loops. In Section 4, we introduce a new class of representations called blocks
and make a conjectural interpretation of the refined Kac polynomials for quivers with enough loops in terms of the numbers of block representations.
In Appendix I, we prove Theorem \ref{hua thm} for the sake of completeness.
In Appendix II, we show some examples of the Kac polynomials and the refined Kac functions.

\section{Definition of the refined Kac functions}

Let $\mathcal{P}$ be the set of partitions of all non-negative integers. For a partition $\lambda=[\lambda_1, \lambda_2, \cdots, \lambda_r]\in \mathcal{P}$ with
 $r > 0$ and $\lambda_i\ge\lambda_{i+1}$ for $1\le i \le r-1$ ,
$|\lambda| := \lambda_1+\cdots+\lambda_r$ is called the \textit{weight} of $\lambda$, $m_{\lambda}^{(k)} := |\{\lambda_i : \lambda_i =  k , i\ge 1\}|$ is called the \textit{mulitplicity} of $\lambda$ of order $k$,
and $\lambda^\mathrm{v} := (m_{\lambda}^{(1)},\cdots,m_{\lambda}^{(r)})$ is called the \textit{mulitplicity vector} of $\lambda$ .
For an $n$-tuple of partitions $\lambda_*=(\lambda^1, \lambda^2, \cdots,\lambda^n) \in \mathcal{P}^n$, $\lambda_*^{(k)}:=(m_{\lambda^1}^{(k)}, m_{\lambda^2}^{(k)}, \cdots, m_{\lambda^n}^{(k)}) \in\mathbb{N}^n$ is called the \textit{multiplicity vector}
of $\lambda_*$ of order $k$ and $\lambda_*^\mathrm{v} :=  (\lambda_*^{(1)}, \lambda_*^{(2)}, \cdots, \lambda_*^{(r)})$ is called the \textit{tuple of multiplicity vectors} of $\lambda_*$, 
where $r$ is the maximal part of all parts in $\lambda^i$ for $1\le i \le n$. Conversely, for each $r$-tuple of vectors $\alpha_* = (\alpha^1, \alpha^2,\cdots,\alpha^r)$ with $\alpha^i \in \mathbb{N}^n$, there exists a unique $n$-tuple of partitions
 $\lambda_*=(\lambda^1, \lambda^2, \cdots,\lambda^n) \in \mathcal{P}^n$ such that  $\lambda_*^\mathrm{v} = \alpha_*$.

Let $\mathbb{Q}(q)[[X_{ik}]]_{1\le i \le n, k\ge 1}$ be the formal power series ring over $\mathbb{Q}(q)$ in infinitely many variables $X_{ik}$. 
For a tuple of dimension vectors $\alpha_*=(\alpha^1,\alpha^2,\cdots, \alpha^r)$ where $\alpha^k=(\alpha^k_1, \cdots, \alpha^k_n) \in \mathbb{N}^n$ for $1\le k \le r$, 
we write $X^{\alpha_*} = \prod_{i=1}^n\prod_{ k=1}^r X_{ik}^{k\alpha^k_i}$. 

Be reminded that, throughout this paper, the superscripts in $\lambda^k$, $\alpha^k$, $\beta^k$, $\alpha_i^k$ and $v_i^k$ are not power indices, they are just labels.

\begin{table}[h]
\parbox{.45\linewidth} {
\centering
\textbf{$X_{ik}$ distributed over $\Gamma$}
\begin{tabular}{c|cccc}
\toprule
\multicolumn{1}{c}{} & \multicolumn{4}{c}{\textrm{Vertex $i$}}  \\
\cmidrule(rl){2-5} 
\textrm{Part $k$} & {$1$} & {$2$} & {$\cdots$} & {$n$}  \\
\midrule
$1$ & $X_{11}$ & $X_{21}$ & $\cdots$ & $X_{n1}$ \\
$2$ & $X_{12}$ & $X_{22}$ & $\cdots$ & $X_{n2}$ \\
$\vdots$ & $\vdots$ & $\vdots$ & $\vdots$ & $\vdots$  \\
$m$ & $X_{1m}$ & $X_{2m}$ & $\cdots$ & $X_{nm}$ \\
$\vdots$ & $\vdots$ & $\vdots$ & $\vdots$ & $\vdots$  \\
\bottomrule
\end{tabular}
}
\hfill
\parbox{.45\linewidth} {
\centering
\textbf{$\alpha_*$ distributed over $\Gamma$}
\begin{tabular}{c|cccc}
\toprule
\multicolumn{1}{c}{} & \multicolumn{4}{c}{\textrm{Vertex $i$}}  \\
\cmidrule(rl){2-5} 
\textrm{Part $k$} & {$1$} & {$2$} & {$\cdots$} & {$n$}  \\
\midrule
$1$ & $\alpha_1^1$ & $\alpha_2^1$ & $\cdots$ & $\alpha_n^1$ \\
$2$ & $\alpha_1^2$ & $\alpha_2^2$ & $\cdots$ & $\alpha_n^2$ \\
$\vdots$ & $\vdots$ & $\vdots$ & $\vdots$ & $\vdots$  \\
$m$ & $\alpha_1^m$ & $\alpha_2^m$ & $\cdots$ & $\alpha_n^m$ \\
$\vdots$ & $\vdots$ & $\vdots$ & $\vdots$ & $\vdots$  \\
\bottomrule
\end{tabular}
}
\end{table}

We define a formal power series $Q_\Gamma$ in $\mathbb{Q}(q)[[X_{ik}]]$ as follows:
\begin{equation}\label{def Q}
Q_\Gamma(q, X_{ik})_{1\le i \le n, k\ge 1}:= 1 + \sum_{\alpha_*} \!
\Big(\prod_{k\ge 1} \frac{q^{\langle \alpha^k, \alpha^k\rangle}} {q^{\langle \beta^k, \beta^k \rangle}}  \frac{|\mathrm{R}_\Gamma(\alpha^k, \mathbb{F}_q)|}{|\mathrm{GL}(\alpha^k, \mathbb{F}_q)|}\Big) X^{\alpha_*},
\end{equation}
where the sum runs over all tuples of dimension vectors  $\alpha_*=(\alpha^1,\alpha^2,\cdots, \alpha^r)$ such that $r\ge 1$, $\alpha^k\in\mathbb{N}^n$ for $1\le k \le r$ with $\alpha^r\ne 0$, 
and $\beta^k=\sum_{i\ge k} \alpha^i$ for $1\le k \le r$.

\begin{dfn} \label{def refine}
The refined Kac functions $A_\Gamma(\lambda_*,q)$ for $\lambda_*\in\mathcal{P}^n\backslash\{0\}$ are defined by the following equation:
\begin{equation}\label{def refined kac log}
(q-1)\mathrm{Log}(Q_\Gamma(q, X_{ik})_{1\le i \le n, k\ge 1}) = \sum_{\lambda_*\in\mathcal{P}^n\backslash\{0\}} A_\Gamma(\lambda_*,q) X^{\lambda^{\mathrm{v}}_*},
\end{equation}
where the $\mathrm{Log}$ is the plethystic logarithm map (see Mozgovoy \cite{SM 2007} for its definition and properties).
\end{dfn}

Note that $A_\Gamma(\lambda_*,q)$ are well defined and independent of the orientations of $\Gamma$. Also note that equation (\ref{def refined kac log}) is equivalent to the following equation because $\mathrm{Log}$ and $\mathrm{Exp}$ are inverse of each other:
\begin{equation}\label{def refined kac exp}
Q_\Gamma(q, X_{ik})_{1\le i \le n, k\ge 1} =  \mathrm{Exp} \Big(\frac{1}{q-1}\sum_{\lambda_*\in\mathcal{P}^n\backslash\{0\}} A_\Gamma(\lambda_*,q) X^{\lambda^{\mathrm{v}}_*}\Big).
\end{equation}

So far we are only certain that $A_\Gamma(\lambda_*,q)\in \mathbb{Q}(q)$. It will be clear in the next section that $A_\Gamma(\lambda_*,q)$ are polynomials in $q$ with non-negative integer coefficients if $\Gamma$ has enough loops. 
Note that if $\Gamma$ does not have enough loops, then $A_\Gamma(\lambda_*,q)$ are not necessarily polynomials (see Appendix II for a counter example).

For $\alpha=(\alpha_1,\cdots,\alpha_n)\in\mathbb{N}^n$, let 
$\Lambda_\alpha = \{ (\lambda^1, \cdots,\lambda^n) \in \mathcal{P}^n : |\lambda^i| = \alpha_i  \textrm{ for all } i \}$. 
By specializing $X_{ik}$ to $X_i$ for all $1\le i\le n$ and $k\ge 1$ in equation (\ref{def Q}) and (\ref{def refined kac exp}), we get the following result by 
of identity (\ref{def PX}) and (\ref{hua formula}).

\begin{prop} 
With the notation above, we have the following identity:
\begin{equation}
A_\Gamma(\alpha,q) = \sum_{\lambda_*\in \Lambda_\alpha} A_\Gamma(\lambda_*,q).
\end{equation}
\end{prop} 

For a representation $M$ of $\Gamma$ over a field, its endomorphism ring is denoted by $\mathrm{End}(M)$. Any nilpotent matrix is conjugate to a direct sum of Jordan matrices with eigenvalues 0, thus the 
sizes of the Jordan blocks give rise to a unique partition. For example, the following two nilpotent matrices give rise to partition $[3,2]$ and $[2,2,1]$ respectively:
$$
\left[
\begin{array}{ccccc}
0 & 1 &  &  &  \\ 
 & 0 & 1 &  &  \\ 
  &  &  0 &  &  \\
 & &  & 0 & 1 \\
 &  &  &  & 0 
\end{array}
\right], \,\,\,\,
\left[
\begin{array}{ccccc}
0 & 1 &  &  &  \\ 
 & 0 &  &  &  \\ 
  &  &  0 & 1  &  \\
 & &  &  0 &  \\
 &  &  &  & 0 
\end{array}
\right].
$$

And hence, every nilpotent endomorphism 
$\theta\in\mathrm{End}(M)$ gives rise to a unique $n$-tuple of partitions $\theta_*\in\mathcal{P}^n$ indexed by the vertices of $\Gamma$.
Let $\lambda_*=(\lambda^1,\cdots,\lambda^n)\in\mathcal{P}^n$ be an $n$-tuple of partitions, $|\lambda_*| : =(|\lambda^1|,\cdots, |\lambda^n|)\in\mathbb{N}^n$ be the dimension vector associated to
$\lambda_*$, and define the set $\mathrm{E}_\Gamma(\lambda_*, \mathbb{F}_q)$ as follows:
$$
\mathrm{E}_\Gamma(\lambda_*, \mathbb{F}_q) = \big\{ (M,\theta)\,|\,M \in \mathrm{R}_\Gamma(|\lambda_*|, \mathbb{F}_q), \theta\in\mathrm{End}(M), \theta\textit{ is nilpotent}, \theta_* = \lambda_*  \big\}.
$$
The finite linear group $\mathrm{GL}(|\lambda_*|, \mathbb{F}_q)$ acts on $\mathrm{E}_\Gamma(\lambda_*, \mathbb{F}_q)$ naturally and the isomorphism classes of
$\mathrm{E}_\Gamma(\lambda_*, \mathbb{F}_q)$ are in one-to-one correspondence with
the orbits. The following result is a consequence of Lemma 2.2 of Bozec-Schiffmann-Vasserot \cite{B-S-V 2018}.

\begin{prop} 
With the notation above, we have the following identity for any $\lambda_*\in\mathcal{P}^n$:
\begin{equation}\label{endormphism module GL}
\frac{|\mathrm{E}_\Gamma(\lambda_*, \mathbb{F}_q)|} {|\mathrm{GL}(|\lambda_*|, \mathbb{F}_q)|} =
\prod_{k=1}^r \frac{q^{\langle \alpha^k, \alpha^k\rangle}} {q^{\langle \beta^k, \beta^k \rangle}}  \frac{|\mathrm{R}_\Gamma(\alpha^k, \mathbb{F}_q)|}{|\mathrm{GL}(\alpha^k, \mathbb{F}_q)|},
\end{equation}
where $\alpha^k\in\mathrm{N}^n$ for $1\le k \le r$ such that  $(\alpha^1, \alpha^2,\cdots,\alpha^r) = \lambda_*^{\mathrm{v}}$, i.e., $\alpha^k = \lambda_*^{(k)}$, and where 
$\beta^k=\sum_{i\ge k} \alpha^i$ for $1\le k \le r$.
\end{prop} 

The following result is a consequence of identity (\ref{def Q}), (\ref{def refined kac exp}) and (\ref{endormphism module GL}).
\begin{prop} 
With the notation above, the following identity holds in the formal power series ring $\mathbb{Q}(q)[[X_{ik}]]_{1\le i \le n, k\ge 1}$ for any quiver $\Gamma$:
\begin{equation}\label{link to CSV}
\sum_{\lambda_*\in\mathcal{P}^n} \frac{|\mathrm{E}_\Gamma(\lambda_*, \mathbb{F}_q)|} {|\mathrm{GL}(|\lambda_*|, \mathbb{F}_q)|}  X^{\lambda^{\mathrm{v}}_*}  =
 \mathrm{Exp}\Big(\frac{1}{q-1} \! \sum_{\lambda_*\in\mathcal{P}^n\backslash\{0\}} \! A_\Gamma(\lambda_*,q) X^{\lambda^{\mathrm{v}}_*}\Big).
\end{equation}
\end{prop}

\section{Positivity of the refined Kac polynomials}

A formal power series $U=\sum_{\alpha\in\mathbb{N}^n} u_\alpha(q) X^\alpha \in\mathbb{Q}(q)[[X_1,\dots,X_n]]$ with constant term $1$ is said to \textit{have the positivity property} if there exist polynomials 
$v_\alpha(q)$ in $q$ with non-negative integer coefficients for all $\alpha\in\mathbb{N}^n\backslash\{0\}$
such that 
\begin{align*}
U = \mathrm{Exp}\Big(\frac{1}{q-1} \sum_{\alpha\in\mathbb{N}^n\backslash\{0\}}\!v_\alpha(q)X^\alpha\Big).
\end{align*}
Note that, if the polynomials $v_\alpha(q)$ exist then they must be unique.
The following theorem is a reformation of Theorem 4.3 of Mozgovoy \cite{SM 2013}.
\begin{thm} \label{M thm}
(Mozgovoy)
Suppose that $C$ is an $n\times n$ matrix with non-negative integer entries, $U=\sum_{\alpha\in\mathbb{N}^n} u_\alpha(q) X^\alpha$ is a formal power series in 
$\mathbb{Q}(q)[[X_1,\dots,X_n]]$ with constant term 1 and 
$V=\sum_{\alpha\in\mathbb{N}^n} u_\alpha(q) q^{\alpha C \alpha^t}X^\alpha$. If $U$ has the positivity property then $V$ must also have the positivity property.
\end{thm}

\begin{lem}\label{positivity lemma}
If $\,\Gamma$ is a quiver with enough loops, then the following formal power series has the positivity property:
\begin{align*}
\sum_{\alpha\in\mathbb{N}^n} \frac{|\mathrm{R}_\Gamma(\alpha, \mathbb{F}_q)|}{|\mathrm{GL}(\alpha, \mathbb{F}_q)|} X^\alpha.
\end{align*}
\end{lem}
\begin{proof}
One form of the Heine formula is the following identity in $\mathbb{Q}(q)[[X]]$:
$$
\sum_{m\ge 0} \frac {X^m} { (1-q)(1-q^2)\cdots (1-q^m)} = \mathrm{Exp}\Big(\frac{1} {1-q} X\Big).
$$
Substituting $q$ with $q^{-1}$ in the above identity, we get
\begin{equation}\label{Heine}
\sum_{m\ge 0} \frac {q^{m^2}} {|\mathrm{GL}(m,\mathbb{F}_q)|} X^m = \mathrm{Exp}\Big(\frac{q} {q-1} X\Big),
\end{equation}
where $|\mathrm{GL}(m,\mathbb{F}_q)|$ is treated as a polynomial function in $q$. Thus,
\begin{align*}
\sum_{\alpha\in\mathbb{N}^n} \frac{q^{\alpha \alpha^t}}{|\mathrm{GL}(\alpha, \mathbb{F}_q)|} X^\alpha &= 
	\sum_{(m_1,\cdots,m_n)\in\mathbb{N}^n}\! \frac{q^{m_1^2+\cdots+m_n^2}}{|\mathrm{GL}(m_1, \mathbb{F}_q)| \cdots |\mathrm{GL}(m_n, \mathbb{F}_q)|} X_1^{m_1} \cdots  X_n^{m_n} \\
&= \prod_{i=1}^n\Big(\sum_{m_i\ge 0} \frac{q^{m_i^2}} {|\mathrm{GL}(m_i,\mathbb{F}_q)|} X_i^{m_i}\Big) \\
&= \prod_{i=1}^n \mathrm{Exp}\Big(\frac{q} {q-1} X_i\Big) \\
&= \mathrm{Exp}\Big(\frac{q} {q-1} \big(X_1+ \cdots + X_n\big)\Big).
\end{align*}
Thus the series on the left hand side has the positivity property. Recall that $C(\Gamma) = (a_{ij})_{1\le i, j  \le n}$ is the companion matrix of $\Gamma$, where $a_{ij}$ is the number of arrows
from vertex $v_i$ to vertex $v_j$. We have
\begin{align*}
\sum_{\alpha\in\mathbb{N}^n} \frac{|\mathrm{R}_\Gamma(\alpha, \mathbb{F}_q)|}{|\mathrm{GL}(\alpha, \mathbb{F}_q)|} X^\alpha &= 
	\sum_{\alpha\in\mathbb{N}^n} \frac{q^{\alpha C(\Gamma) \alpha^t}}{|\mathrm{GL}(\alpha, \mathbb{F}_q)|} X^\alpha \\
&= \sum_{\alpha\in\mathbb{N}^n} \frac{q^{\alpha \alpha^t}}{|\mathrm{GL}(\alpha, \mathbb{F}_q)|}  q^{\alpha (C(\Gamma)-I) \alpha^t}X^\alpha.
\end{align*}
Since $\Gamma$ has enough loops, $C(\Gamma)-I$ must be a matrix with non-negative integer entries. This lemma now follows from Theorem \ref{M thm}.
\end{proof}

For any positive integer $m$, let $\mathcal{P}_{m}$ be the set of all partitions whose parts are less than or equal to $m$.

\begin{lem}\label{width 1}
We have following identity for any finite quiver  $\Gamma$:
\begin{equation}
\sum_{\alpha\in\mathbb{N}^n} \frac{|\mathrm{R}_\Gamma(\alpha, \mathbb{F}_q)|}{|\mathrm{GL}(\alpha, \mathbb{F}_q)|} X^\alpha
= \mathrm{Exp} \Big(\frac{1}{q-1} \sum_{\lambda_*\in(\mathcal{P}_1)^n} \!\!A_\Gamma(\lambda_*,q) X^{\lambda_*^\mathrm{(1)}} \Big).
\end{equation}
If  $\,\Gamma$ has enough loops, then $A_\Gamma(\lambda_*,q)$ are polynomials in $q$ with non-negative integer coefficients
for all $\lambda_*\in(\mathcal{P}_1)^n$.
\end{lem}
\begin{proof}
The first statement is a consequence of equation  (\ref{def Q}) and (\ref{def refined kac exp}) by setting $X_{ik}=0$ for $1\le i\le n$ and $k\ge 2$
and replacing $X_{i1}$ by $X_i$ for $1\le i \le n$. 
The second statement is a consequence of Lemma \ref{positivity lemma}.
\end{proof}

For a fixed positive integer $m$, we define a formal power series $Q^m_{\Gamma}(q, X_{ik})$ by letting $X_{ik}=0$ for all $1\le i \le n$ and $k > m$ in equation (\ref{def Q}). Thus we have
\begin{equation}\label{def Q_Gamma^m}
Q^m_{\Gamma}(q, X_{ik})_{1\le i \le n, 1\le k \le m} = \sum_{\alpha_*} \!
\Big(\prod_{k=1}^m \frac{q^{\langle \alpha^k, \alpha^k\rangle}} {q^{\langle \beta^k, \beta^k \rangle}}  \frac{|\mathrm{R}_\Gamma(\alpha^k, \mathbb{F}_q)|}{|\mathrm{GL}(\alpha^k, \mathbb{F}_q)|}\Big) X^{\alpha_*},
\end{equation}
where the sum runs over all tuples of dimension vectors  $\alpha_*=(\alpha^1,\alpha^2,\cdots, \alpha^m)$, $\alpha^k\in\mathbb{N}^n$ for $1\le k \le m$, 
and $\beta^k=\sum_{i\ge k} \alpha^i$ for $1\le k \le m$.

Since $\langle \alpha^k, \alpha^k\rangle =  \alpha^k (I-C(\Gamma)) (\alpha^k)^t$, $\langle \beta^k, \beta^k\rangle =  \beta^k (I-C(\Gamma)) (\beta^k)^t$ and 
$|\mathrm{R}_\Gamma(\alpha^k, \mathbb{F}_q)|$ equals $q^{\alpha^k C(\Gamma) (\alpha^k)^t}$, we have
\begin{equation}\label{quadrtic}
\log_q\!\Big(\prod_{k=1}^m \frac{q^{\langle \alpha^k, \alpha^k\rangle}} {q^{\langle \beta^k, \beta^k \rangle}}  |\mathrm{R}_\Gamma(\alpha^k, \mathbb{F}_q)|\Big) 
= \sum_{k=1}^m\alpha^k (\alpha^k)^t + \beta^k (C(\Gamma)-I) (\beta^k)^t.
\end{equation}
Suppose that $\alpha^k = (\alpha_1^k, \cdots, \alpha_n^k) \in \mathbb{N}^n$ for $1\le k \le m$.
It is evident that the right hand side of identity (\ref{quadrtic}) is a quadratic form in entries $\alpha_i^k$ for $1\le i \le n$ and $1\le k \le m$. 

We construct a quiver $\Gamma_{\!m}$ when $\Gamma$ has enough loops as follows: $\Gamma_{\!m}$ has $mn$ vertices $v_i^k $ for $1\le i \le n$ and $1\le k \le m$.
A linear order $\prec$ on the vertices is defined by setting $v_i^k \prec v_s^t$ if either $k < t$ or $k=t$ and $i < s$.
The number of arrows from vertex $v_i^k$ to vertex $v_s^t$ is the coefficient of $\alpha_i^k \alpha_s^t$ in the quadratic form (\ref{quadrtic}) if $v_i^k \preceq v_s^t$, zero otherwise.
This is well defined because $C(\Gamma)-I$ is a matrix with non-negative integer entries.
The term $\sum_{k=1}^m\alpha^k (\alpha^k)^t$ in equation (\ref{quadrtic}) implies that $\Gamma_{\!m}$ has enough loops.
Let $C(\Gamma_{\!m})$ be the companion matrix of $\Gamma_{\!m}$, then we have 
\begin{equation}\label{new matrix}
\sum_{k=1}^m\alpha^k (\alpha^k)^t + \beta^k (C(\Gamma)-I) (\beta^k)^t = \gamma C(\Gamma_{\!m}) \gamma^t,
\end{equation}
where $\gamma = (\alpha_i^k)_\prec$ for $1\le i \le n$ and $1\le k \le m$. 
\begin{table}[h]
\parbox{.45\linewidth} {
\centering
\textbf{Vertex arrangement of $\Gamma _m$}
\begin{tabular}{c|cccc}
\toprule
\multicolumn{1}{c}{} & \multicolumn{4}{c}{\textrm{Vertex $i$ of $\Gamma$}}  \\
\cmidrule(rl){2-5} 
\textrm{Part $k$} & {$1$} & {$2$} & {$\cdots$} & {$n$}  \\
\midrule
$1$ & $v_1^1$ & $v_2^1$ & $\cdots$ & $v_n^1$ \\
$2$ & $v_1^2$ & $v_2^2$ & $\cdots$ & $v_n^2$ \\
$\vdots$ & $\vdots$ & $\vdots$ & $\vdots$ & $\vdots$  \\
$m$ & $v_1^m$ & $v_2^m$ & $\cdots$ & $v_n^m$ \\
\bottomrule
\end{tabular}
}
\hfill
\parbox{.45\linewidth} {
\centering
\textbf{$\alpha_*$ distributed over $\Gamma _m$}
\begin{tabular}{c|cccc}
\toprule
\multicolumn{1}{c}{} & \multicolumn{4}{c}{\textrm{Vertex $i$ of $\Gamma$}}  \\
\cmidrule(rl){2-5} 
\textrm{Part $k$} & {$1$} & {$2$} & {$\cdots$} & {$n$}  \\
\midrule
$1$ & $\alpha_1^1$ & $\alpha_2^1$ & $\cdots$ & $\alpha_n^1$ \\
$2$ & $\alpha_1^2$ & $\alpha_2^2$ & $\cdots$ & $\alpha_n^2$ \\
$\vdots$ & $\vdots$ & $\vdots$ & $\vdots$ & $\vdots$  \\
$m$ & $\alpha_1^m$ & $\alpha_2^m$ & $\cdots$ & $\alpha_n^m$ \\
\bottomrule
\end{tabular}
}
\end{table}

For any $r\in\mathbb{N}$, let $[1^r]$ be the partition which has $r$ parts with each part equal to $1$. Recall that for a partition $\lambda\in\mathcal{P}$, $m_\lambda^{(k)}$
is the number of parts equal to $k$ in $\lambda$. 
Define a map $\tau_m$ as follows:
\begin{align*}
\tau_m: \hspace{9 mm} (\mathcal{P}_m)^n \hspace{3 mm}&\to \hspace{3 mm} (\mathcal{P}_1)^{mn} \\
 (\lambda^1,\cdots,\lambda^n) \hspace{3 mm} & \mapsto \hspace{3 mm} ([1^{\alpha_i^k}])_{\prec}, \textrm{ where $\alpha_i^k = m_{\lambda^i}^{(k)}$.}
\end{align*}
For example, when $n=3$ and $m=2$ we have:
$$
\tau_2(([2,1,1],[2,2,1],[2,2,2])) = ([1,1],[1],[0],[1],[1,1],[1,1,1]).
$$
In the notation of multiplicity vectors, $\tau_2: ((2,1,0), (1,2,3)) \mapsto (2,1,0, 1,2,3)$. It is evident that $\tau_m$ is a one-to-one map.

\begin{thm}\label{positivity thm}
With the notation above and assuming that $\Gamma$ has enough loops, we have the following identity for any $\lambda_*\in(\mathcal{P}_m)^n$:
$$
A_\Gamma(\lambda_*, q) = A_{\Gamma_{\!m}}\!(\tau_m(\lambda_*),q).
$$
And hence, for any $\lambda_*\in\mathcal{P}^n$, $A_\Gamma(\lambda_*, q)$ is a polynomial in $q$ with non-negative integer coefficients.
\end{thm}

\begin{proof}
It follows from equation (\ref{def Q_Gamma^m}) (\ref{quadrtic}) and (\ref{new matrix}) that
\begin{align*}
Q^m_{\Gamma}(q, X_{ik})_{1\le i \le n, 1\le k \le m} &= \sum_{\alpha=(\alpha_i^k)_{\prec}} \!\! \frac{q^{\alpha C(\Gamma_{\!m}) \alpha^t}}{|\mathrm{GL}(\alpha, \mathbb{F}_q)|}
\prod_{i=1}^n\prod_{k=1}^m X_{ik}^{k\alpha_i^k} \\
&=\sum_{\alpha=(\alpha_i^k)_{\prec}} \!\! \frac{|\text{R}_{\Gamma_{\!m}}\!(\alpha, \mathbb{F}_q)|}{|\mathrm{GL}(\alpha, \mathbb{F}_q)|}
\prod_{i=1}^n\prod_{k=1}^m X_{ik}^{k\alpha_i^k} \\
&\overset{(\ref{def Q_Gamma^m})}{=} Q^1_{\Gamma_{\!m}}\!(q, X_{ik})_{1\le i \le n, 1\le k \le m \,|\/ X_{ik} = X_{ik}^k} .
\end{align*}

It follows from Definition \ref{def refine} that
\begin{align*}
(q-1)\mathrm{Log}\big(Q^m_{\Gamma}(q, X_{ik})_{1\le i \le n, 1\le k \le m}\big) &= \sum_{\lambda_*\in (\mathcal{P}_m)^n} \!A_\Gamma(\lambda_*,q) X^{\lambda_*^{\mathrm{v}}} ,\\
(q-1)\mathrm{Log}\big(Q^1_{\Gamma_{\!m}}\!(q, X_{ik})_{1\le i \le n, 1\le k \le m}\big) &=  \sum_{\lambda_*\in (\mathcal{P}_1)^{mn}} \!A_{\Gamma_{\!m}}\!(\lambda_*,q) X^{\lambda_*^{\mathrm{(1)}}}.
\end{align*}
Thus we have 
$$ 
\sum_{\lambda_*\in (\mathcal{P}_m)^n} \!\!A_\Gamma(\lambda_*,q) X^{\lambda_*^{\mathrm{v}}} 
=  \sum_{\lambda_*\in (\mathcal{P}_1)^{mn}} \!\!A_{\Gamma_{\!m}}\!(\lambda_*,q) X^{\lambda_*^{\mathrm{(1)}}} | _{X_{ik} = X_{ik}^k}.
$$
By comparing the coefficients of $X^{\lambda_*^{\mathrm{v}}}$ for $\lambda_*\in (\mathcal{P}_m)^n$ on both sides, we have
$$
A_\Gamma(\lambda_*, q) = A_{\Gamma_{\!m}}\!(\tau_m(\lambda_*),q).
$$
Since $\Gamma_{\!m}$ has enough loops and $\tau_m(\lambda_*)\in(\mathcal{P}_1)^{mn}$, Lemma \ref{width 1} implies that 
$A_\Gamma(\lambda_*, q)$ are polynomials in $q$ with non-negative integer coefficients for all $\lambda_*\in (\mathcal{P}_m)^n$.
Since $m$ can be any positive integer, the same conclusion must be true for any $\lambda_*\in\mathcal{P}^n$.
\end{proof}

It is worth noticing that Theorem \ref{positivity thm} does not depend on the validity of Kac Conjecture 1.1.

Let $g$ be a positive integer and $\Gamma$ the $g$-loop quiver whose companion matrix is $[\,g\,]$.  As $\langle(a),(a)\rangle = (1-g)a^2$ and $|\mathrm{R}_{\Gamma}(a, \mathbb{F}_q)| = q^{ga^2}$ for $a\in \mathbb{N}$, 
the quiver $\Gamma_2$ can be found as follows:
\begin{align*}
&\hspace{6mm}Q^2_{\Gamma}(q, X_{11}, X_{12})  \\
= & 
\sum_{(m_1,m_2)\in\mathbb{N}^2} \frac{q^{(1-g)m_1^2}}{q^{(1-g)(m_1+m_2)^2}} \frac{q^{(1-g)m_2^2}}{q^{(1-g)m_2^2}} 
\frac{q^{gm_1^2}}{|\mathrm{GL}(m_1, \mathbb{F}_q)|} \frac{q^{gm_2^2}}{|\mathrm{GL}(m_2, \mathbb{F}_q)|} X_{11}^{m_1}X_{12}^{2m_2} \\
= &
\sum_{(m_1,m_2)\in\mathbb{N}^2} \frac{q^{gm_1^2 + (2g-1)m_2^2 + 2(g-1) m_1m_2}}{|\mathrm{GL}((m_1,m_2), \mathbb{F}_q)|} X_{11}^{m_1}X_{12}^{2m_2}.
\end{align*}
Thus we have the following:
\begin{center} \label{table5}
\def\arraystretch{1.5}
\begin{tabular}{c|c|c} 
	\hline
	 Companion matrix of $\Gamma_2$ & $\text{\hspace{3 mm}}$ Quiver $\Gamma$ ($g=2$) $\text{\hspace{3 mm}}$ 
	 & $\text{\hspace{3 mm}}$ Quiver $\Gamma_2$ ($g=2$) $\text{\hspace{3 mm}}$ \\  
	\hline
	$\begin{bmatrix}
	g & 2(g-1)\\
	0 & 2g-1
	\end{bmatrix}$
	& 
	\begin{tikzcd}
	\bullet \arrow[out=120,in=60,loop,swap] \arrow[out=240,in=300,loop,swap] 
	\end{tikzcd} 
	& 
	\begin{tikzcd}
	\bullet \arrow[r, yshift=-0.4ex] \arrow[r, yshift=0.4ex] \arrow[out=120,in=60,loop,swap] \arrow[out=240,in=300,loop,swap] & 
	\bullet \arrow[out=30,in=-30,loop,swap] \arrow[out=150,in=90,loop,swap] \arrow[out=210,in=270,loop,swap]
	\end{tikzcd}
	\\
	\hline
\end{tabular}
\end{center}

Computer programs have confirmed the following relationships:
\begin{align*}
&A_\Gamma([1],q) &=& A_{\Gamma_2}(([1],[0]),q) \!\!\!\!\! &=&\, q^{g}, \\
&A_\Gamma([1,1],q) &=& A_{\Gamma_2}(([1,1],[0]),q) \!\!\!\!\! &=& \,\frac{q^{2g+1}(q^{2g-2} - 1)} {q^2-1}, \\
&A_\Gamma([2],q) &=& A_{\Gamma_2}(([0],[1]),q) \!\!\!\!\! &=& \, q^{2g-1}, \\
&A_\Gamma([2,1],q) &=& A_{\Gamma_2}(([1],[1]),q) \!\!\!\!\! &=&  \, \frac{q^{3g-1}(q^{2g-2} - 1)} {q-1}, \\
&A_\Gamma([2,2],q) &=& A_{\Gamma_2}(([0],[1,1]),q) \!\!\!\!\! &=&  \, \frac{q^{4g-1}(q^{4g-4} - 1)} {q^2-1}, \\
&A_\Gamma([2,1,1],q)\!\!\!\!\!\!&=& A_{\Gamma_2}(([1,1],[1]),q) \!\!\!\!\! &=& \,\frac{q^{4g-1}(q^{6g-5} -2q^{2g-1} -q^{2g-2} +q+1)} {(q^2-1)(q-1)}. \\
\end{align*}
When $g=1$, $\Gamma$ is known as the Jordan quiver and its representations are completely classified. We have 
\[
A_\Gamma(\lambda,q) = \begin{cases}
q & \text{if $\lambda = [n]$ for $n\ge 1$,}\\
0 & \text{otherwise.} \\
\end{cases}
\]

\section{A conjectural interpretation of the refined Kac polynomials}
A representation $M$ of a quiver $\Gamma$ over a field $\mathbb{F}$ is called a \textit{brick} if its endomorphism algebra $\textrm{End}_\mathbb{F}(M)$ is a division algebra. 
Bricks are also known as Schur representations. A brick $M$ is called an \textit{absolute brick} if $\textrm{End}_\mathbb{F}(M)$ is isomorphic to $\mathbb{F}$. 
Thus a brick is alway indecomposable and an absolute brick is always absolutely indecomposable.

Let $\overline{\mathbb{F}}$ be the algebraic closure of $\mathbb{F}$. Then all representations of $\Gamma$ over $\overline{\mathbb{F}}$ with a given dimension vector $\alpha$ form an affine
variety over $\overline{\mathbb{F}}$ and all bricks of $\Gamma$ over $\overline{\mathbb{F}}$ with dimension vector $\alpha$ form a constructible set. The latter is denoted by 
$V_\Gamma(\alpha, \overline{\mathbb{F}})$. 

A representation $M$ of $\Gamma$ over $\overline{\mathbb{F}}$ with dimension vector $\alpha$ 
is called a \textit{block} if $M$ is indecomposable and its orbit lies in the closure of $V_\Gamma(\alpha, \overline{\mathbb{F}})$ under the Zariski topology. A representation $M$ of a quiver $\Gamma$ over a field $\mathbb{F}$ is called 
an \textit{absolute block} if $M\otimes_\mathbb{F} \overline{\mathbb{F}}$ is a block of $\Gamma$ over $\overline{\mathbb{F}}$. An absolute block is apparently absolutely indecomposable. By the definition of blocks, all blocks of 
a quiver over an algebraically closed field with a given dimension vector form an algebraic variety.

For example, let $\Gamma$ be the 2-loops quiver $\circlearrowleft \!\!\!\bullet \!\!\! \circlearrowright$ and $V$ be the following set, where all parameters belong to $\overline{\mathbb{F}}$:
\begin{align*}
	\Bigg\{         
	g^{-1}\Bigg(\begin{bmatrix}
	\lambda & 0\\
	0 & \mu
	\end{bmatrix},
	\begin{bmatrix}
	\sigma & \tau\\
	\delta & \eta
	\end{bmatrix} \Bigg) g \, \Big |
	\lambda \ne \mu, \tau \ne 0 \textrm{ or }\delta \ne 0, g\in\mathrm{GL}(2,\mathbb{\overline{\mathbb{F}}})
	\Bigg\}.
\end{align*}
Then every representation in $V$ is a brick of $\Gamma$ over $\overline{\mathbb{F}}$ because its endomorphism algebra is isomorphic to $\overline{\mathbb{F}}$. And hence, every representation of the following form 
\begin{align*}     
	\Bigg(\begin{bmatrix}
	\lambda & 0\\
	0 & \lambda
	\end{bmatrix},
	\begin{bmatrix}
	\sigma & 1\\
	0 & \sigma
	\end{bmatrix} \Bigg)_{\lambda,\sigma\in\overline{\mathbb{F}}}
\end{align*}
is a block of $\Gamma$ over $\overline{\mathbb{F}}$, because it is indecomposable and its orbit lies in the closure of $V$.

Let $B_\Gamma(\alpha, q)$ be the number isomorphism classes of absolute blocks of $\Gamma$ over the finite field $\mathbb{F}_q$ with dimension vector $\alpha$.
\begin{cjc}\label{conj}
With the notation above and assuming that $\Gamma$ has enough loops, the following formal identity holds:
\begin{align*}
\sum_{\alpha\in\mathbb{N}^n} \frac{|\mathrm{R}_\Gamma(\alpha, \mathbb{F}_q)|}{|\mathrm{GL}(\alpha, \mathbb{F}_q)|} X^\alpha
= \mathrm{Exp} \Big(\frac{1}{q-1} \sum_{\alpha\in\mathbb{N}^n\backslash\{0\}} \!\!B_\Gamma(\alpha,q) X^{\alpha} \Big).
\end{align*}
\end{cjc}
If $\Gamma$ is the Jordan quiver, then the Jordan normal form theorem implies that all bricks of $\Gamma$ over $\overline{\mathbb{F}}_q$ are one dimensional, and hence blocks and bricks of $\Gamma$ over $\overline{\mathbb{F}}_q$ are identical. 
We have $B_\Gamma(1, q)=q$. Thus, for the Jordan quiver Conjecture \ref{conj} is equivalent to the Heine formula (\ref{Heine}). 
If Conjecture \ref{conj} is true in general, then Lemma \ref{width 1} implies that $B_\Gamma(\alpha,q)$ are polynomials in $q$ with non-negative integer coefficients.
\begin{cor} With the notation above and assuming that Conjecture \ref{conj} is true and $\Gamma$ has enough loops, we have the following identity for any positive integer $m$ and 
any $\lambda_*\in(\mathcal{P}_m)^n$:
\begin{align*}
A_\Gamma(\lambda_*, q) = B_{\Gamma_{\!m}}\!(\tau_m(\lambda_*),q).
\end{align*}
\end{cor}
Theorem \ref{positivity thm} and Conjecture \ref{conj} suggest that the study of representations or at least the study of numbers of representations of a quiver with enough loops 
can be reduced to the study of blocks of quivers with enough loops and it can be further reduced to the study of bricks of quivers with enough loops.

\vspace{0.2cm}
\textbf{\large{Acknowledgments}}
\vspace{0.1cm}

Special thanks are due to my PhD thesis advisors Peter Donovan and Jie Du for leading me into the wonderful world of representation theory of quivers over finite fields. My sincere thanks go to 
Shaoxue Liu, Yingbo Zhang, Jie Xiao, Bangming Deng and Xueqing Chen for their endless supports, encouragements and friendships over the years.

\vspace{5mm}
\textbf{\large{Appendix I. Proof of Theorem \ref{hua thm}} }
\vspace{2mm}

First, we recall some notation from Hua \cite{JH 2000}. For a partition $\lambda = [\lambda_1,\lambda_2,\cdots] \in \mathcal{P}$, let $\lambda' = [\lambda'_1, \lambda'_2,\cdots] \in \mathcal{P}$ be its conjugate partition. Thus, 
$\lambda'_k = \sum_{i\ge k}m_\lambda^{(i)}$ for $k\ge 1$. Let $\mu=[\mu_1, \mu_2, \cdots]\in \mathcal{P}$ be another partition and $\mu'=[\mu'_1, \mu'_2, \cdots]\in \mathcal{P}$ be its conjuagte partition, 
the ``inner product'' of $\lambda$ and $\mu$ is defined as follows:
\begin{equation}\label{partition inner prod}
\langle\lambda, \mu\rangle = \sum_{i\ge 1}\lambda_i'\mu_i'.
\end{equation}
Let $\varphi_m(q)=(1-q)(1-q^2)\cdots(1-q^m)$ for $m\ge 1$ and $\varphi_0(q)=1$. For $\lambda\in\mathcal{P}$, let $b_\lambda(q) = \prod_{i\ge1}\varphi_{m_i}(q)$ where 
$m_i = m_\lambda^{(i)}$ for $i\ge 1$.

The formal power series $P_\Gamma(X_1,\cdots,X_n,q)$ appeared in Hua \cite{JH 2000} is defined as follows:
\begin{equation}\label{def PX original}
P_\Gamma(X_1,\cdots,X_n,q) := \sum_{\pi_1,\cdots,\pi_n\in\mathcal{P}} \!
\frac{q^{\sum_{1\le i, j \le n}\!a_{ij}\langle\pi_i, \pi_j\rangle} }
{\prod_{1\le i \le n}\!q^{\langle \pi_i, \pi_i\rangle} b_{\pi_i}\!(q^{-1})}X_1^{|\pi_1|}\cdots X_n^{|\pi_n|}.
\end{equation}
We define rational functions $J_\Gamma(\alpha,q)\in\mathbb{Q}(q)$ for all $\alpha\in\mathbb{N}^n\backslash\{0\}$ as follows:
\begin{equation}\label{new def H} 
\log\left(P_\Gamma(X_1,\cdots,X_n,q) \right) = \sum_{\alpha\in\mathbb{N}^n\backslash\{0\}} \!J_\Gamma(\alpha,q)X^\alpha ,
\end{equation}
where the $\log$ is the formal logarithm, i.e., $\log(1+X) = \sum_{i\ge 1} (-1)^{i-1} X^i/i$. Thus we have $J_\Gamma(\alpha,q) = H_\Gamma(\alpha,q)/\bar{\alpha}$,
where $\bar{\alpha}=\gcd(\alpha_1, \cdots, \alpha_n)$ and $H_\Gamma(\alpha,q)$ is the rational function defined on page 1021 of Hua \cite{JH 2000}.

With the notation above, Theorem 4.6 of Hua \cite{JH 2000} is equivalent to the following identity:
\begin{equation}\label{A=J}
A_\Gamma(\alpha,q) = (q-1)\sum_{d\,|\,\bar{\alpha}}\frac{\mu(d)}{d}J_\Gamma\Big(\frac{\alpha}{d}, q^d\Big),
\end{equation}
where the sum runs over all divisors of $\bar{\alpha}$ and $\mu$ is the classical M\"obius function.

By the definition of the plethystic logarithm from Mozgovoy \cite{SM 2007}, we have
\begin{align*}
\mathrm{Log} (P_\Gamma(X_1,\cdots,X_n,q) )&= \sum_{k\ge 1} \frac{\mu(k)}{k} \log ( P_\Gamma(X_1^k,\cdots,X_n^k,q^k) ) \\
&\overset{(\ref{new def H})}{=}\sum_{k\ge 1} \frac{\mu(k)}{k} \sum_{\alpha\in\mathbb{N}^n\backslash\{0\}} \!J_\Gamma(\alpha,q^k)X^{k\alpha} \\
&= \sum_{\alpha\in\mathbb{N}^n\backslash\{0\}} \Big(\sum_{d\,|\,\bar{\alpha}}\frac{\mu(d)}{d}J_\Gamma\Big(\frac{\alpha}{d}, q^d\Big) \Big)X^\alpha \\
&\overset{(\ref{A=J})}{=} \frac{1}{q-1}  \sum_{\alpha\in\mathbb{N}^n\backslash\{0\}} A_\Gamma(\alpha,q)X^\alpha,
\end{align*}
thus, we have 
\begin{equation}
P_\Gamma(X_1,\cdots,X_n,q) = \,\,\text{Exp}\Big(\frac{1}{q-1} \sum_{\alpha\in\mathbb{N}^n\backslash\{0\}}\!A_\Gamma(\alpha,q)X^\alpha\Big).
\end{equation}

Thus, to prove Theorem \ref{hua thm}, it is sufficient to show that $P_\Gamma(X_1,\cdots,X_n, q)$ defined by equation (\ref{def PX original}) is equal to $P_\Gamma(q,X_1,\cdots,X_n)$ defined by equation (\ref{def PX}).
 
We start with  (\ref{def PX original}). Let $\pi_* = (\pi_1,\cdots,\pi_n) \in\mathcal{P}^n$ and $\alpha_* = \pi_*^\textrm{v}$. If we assume that $\alpha_* = (\alpha^1, \alpha^2, \cdots)$ and $\alpha^k = (\alpha_1^k, \cdots, \alpha_n^k) \in\mathbb{N}^n$, then we have
$\alpha_i^k = m_{\pi_i}^{(k)}$ for $1\le i \le n$ and $k\ge 1$.
Let $\beta_* = (\beta^1, \beta^2, \cdots)$ be a tuple of vectors such that $\beta^k = \sum_{s\ge k}\alpha^s$ for $k\ge 1$. If we assume that $\beta^k = (\beta_1^k, \cdots, \beta_n^k) \in\mathbb{N}^n$, then we have 
$\beta_i^k = \sum_{s\ge k}\alpha_i^s$ for $1\le i \le n$ and $k\ge 1$. Thus $[\beta_i^1, \beta_i^2, \cdots] = \pi'_i$ for $1\le i \le n$.
By (\ref{partition inner prod}), we have
$\langle \pi_i, \pi_j \rangle  = \sum_{k\ge 1} \beta_i^k \beta_j^k$ for $1\le i,j \le n$. 

Since
\begin{align*}
\sum_{1\le i,j \le n} a_{ij}\langle \pi_i, \pi_j \rangle - \sum_{1\le i \le n} \langle \pi_i, \pi_i \rangle  &= \sum_{1\le i,j\le n} a_{ij} \sum_{k\ge 1} \beta_i^k \beta_j^k - \sum_{1\le i \le n} \sum_{k\ge 1} \beta_i^k \beta_i^k\\
&= \sum_{k\ge 1} \Big(\sum_{1\le i,j\le n} a_{ij} \beta_i^k \beta_j^k - \sum_{1\le i \le n} \beta_i^k \beta_i^k \Big) \\
&\overset{(\ref{def euler form})}{=} -\sum_{k\ge 1} \langle\beta^k, \beta^k \rangle,
\end{align*}
we have
\begin{equation}\label{euler prod pi}
\frac{q^{\sum_{1\le i, j \le n}\!a_{ij}\langle\pi_i, \pi_j\rangle} } {\prod_{1\le i \le n}\!q^{\langle \pi_i, \pi_i\rangle} } =
\prod_{k \ge 1} \frac{1}{q^{ \langle\beta^k, \beta^k \rangle}}.
\end{equation}
Since
\begin{align*}
\prod_{1\le i \le n} b_{\pi_i}\!(q^{-1}) 
&= \prod_{1\le i \le n}\prod_{k\ge 1} (1-q^{-1})(1-q^{-2})\cdots (1-q^{-{\alpha_i^k}}) \\
&\overset{(\ref{order of GL})}{=} \prod_{1\le i \le n}\prod_{k\ge 1} q^{-(\alpha_i^k)^2} | \mathrm{GL}(\alpha_i^k, \mathbb{F}_q)| \\
&=\Big( \prod_{1\le i \le n}\prod_{k\ge 1} q^{-(\alpha_i^k)^2} \Big)  \Big(\prod_{k\ge 1}|\mathrm{GL}(\alpha^k, \mathbb{F}_q)| \Big), \\
\end{align*}
and 
\begin{align*}
\prod_{k\ge 1} q^{\langle \alpha^k, \alpha^k\rangle} |\mathrm{R}_\Gamma(\alpha^k, \mathbb{F}_q)| &\overset{(\ref{num of rep})}{=} \prod_{k\ge 1} q^{\alpha^k(I-C(\Gamma))(\alpha^k)^t}\cdot q^{\alpha^kC(\Gamma)(\alpha^k)^t} \\
&= \prod_{k\ge 1} q^{\alpha^k(\alpha^k)^t} \\
&= \prod_{k\ge 1} \prod_{1\le i \le n} q^{(\alpha_i^k)^2},
\end{align*}
we have
\begin{equation}\label{equation b_pi}
\frac{1}{\prod_{1\le i \le n} b_{\pi_i}\!(q^{-1})} = \prod_{k\ge 1} q^{\langle \alpha^k, \alpha^k\rangle} \frac{|\mathrm{R}_\Gamma(\alpha^k, \mathbb{F}_q)|} {|\mathrm{GL}(\alpha^k, \mathbb{F}_q)|}.
\end{equation}
Putting (\ref{euler prod pi}) and (\ref{equation b_pi}) together, we have
\begin{align*}
\frac{q^{\sum_{1\le i, j \le n}\!a_{ij}\langle\pi_i, \pi_j\rangle}} {\prod_{1\le i \le n}\!q^{\langle \pi_i, \pi_i\rangle} b_{\pi_i}\!(q^{-1})} 
= \prod_{k\ge 1} \frac{q^{\langle \alpha^k, \alpha^k\rangle}} {q^{\langle \beta^k, \beta^k \rangle}}  \frac{|\mathrm{R}_\Gamma(\alpha^k, \mathbb{F}_q)|}{|\mathrm{GL}(\alpha^k, \mathbb{F}_q)|}.
\end{align*}
This shows that the series $P_\Gamma(X_1,\cdots,X_n, q)$ defined by equation (\ref{def PX original}) is equal to the series $P_\Gamma(q,X_1,\cdots,X_n)$ defined by equation (\ref{def PX}) and this finishes the proof of Theorem \ref{hua thm}.

\vspace{3mm}
\textbf{\large{Appendix II. Examples of the Kac polynomials and the refined Kac functions}}
\vspace{2mm}

\begin{center} \label{table1}
	\def\arraystretch{1.2}
	\begin{tabular}{l|l} 
		\multicolumn{2}{c}{\textbf{Table 1:} Some Kac polynomials for the 2-loop quiver $\circlearrowleft\!\!\!\bullet \!\!\! \circlearrowright$}\\
		\hline
		$n$ & $A_\Gamma(n, q)$ \\
		\hline
		$1$ & $q^2$ \\
		$2$ & $q^5+q^3$ \\
		$3$ & $q^{10}+q^8+q^7+q^6+q^5+q^4$ \\
		$4$ & $q^{17}+q^{15}+q^{14}+2q^{13}+q^{12}+3q^{11}+2q^{10}+4q^9+2q^8+3q^7+q^6+q^5$ \\
		\hline
	\end{tabular}
\end{center}

\vspace{5mm}
\begin{center}\label{table2}
	\def\arraystretch{1.2}
	\begin{tabular}{l|l|l} 
		\multicolumn{3}{c}{\textbf{Table 2:} Some refined Kac polynomials for the 2-loop quiver $\circlearrowleft\!\!\!\bullet \!\!\! \circlearrowright$} \\
		\hline
		$\,\lambda$ & $\lambda^\mathrm{v}$ & $A_\Gamma(\lambda_, q)$ \\
		\hline
		$[1]$ & (1) & $q^2$ \\
		$[2]$ & (0,1) & $q^3$ \\
		$[1,1]$ & (2) & $q^5$ \\
		$[3]$ & (0,0,1) & $q^4$ \\
		$[2,1]$ & (1,1) & $q^6+q^5$ \\
		$[1,1,1]$ & (3) & $q^{10}+q^8+q^7$ \\
		$[4]$ & (0,0,0,1) & $q^5$ \\
		$[3,1]$ & (1,0,1) & $q^7+q^6$ \\
		$[2,2]$ & (0,2) & $q^9+q^7$ \\
		$[2,1,1]$ & (2,1) & $q^{11}+q^{10}+2q^9+2q^8+q^7$ \\
		$[1,1,1,1]$ & (4) & $q^{17}+q^{15}+q^{14}+2q^{13}+q^{12}+2q^{11}+q^{10}+q^9$ $\text{\hspace{6 mm} }$\\
		\hline
	\end{tabular}
\end{center}

\vspace{5mm}
\begin{center} \label{table3}

\def\arraystretch{1.2}
\begin{tabular}{l|l|l} 
	\multicolumn{3}{c} {\textbf{Table 3:} Some Kac polynomials for the quiver $\circlearrowleft\!\!\!\bullet \!\!\! \rightarrow \!\!\! \bullet \!\!\!\circlearrowright$ and $\bullet \!\!\! \rightarrow \!\!\! \bullet$} \\
	\hline
	$\,\alpha$ & $A_\Gamma(\alpha, q)$ for $\circlearrowleft\!\!\!\bullet \!\!\! \rightarrow \!\!\! \bullet \!\!\!\circlearrowright$ $\text{\hspace{20 mm}}$
	& $A_\Gamma(\alpha, q)$ for $\bullet \!\!\! \rightarrow \!\!\! \bullet$ $\text{\hspace{20 mm}}$\\
	\hline
	$(0,1)$ & $q$ & $1$ \\  
	$(0,2)$ & $q$ & $0$ \\  
	$(1,0)$ & $q$ & $1$ \\  
	$(1,1)$ & $q^2$ & $1$ \\  
	$(1,2)$ & $q^3+q^2$ & $0$ \\  
	$(2,0)$ & $q$ & $0$ \\  
	$(2,1)$ & $q^3+q^2$ & $0$ \\  
	$(2,2)$ & $q^5+q^4+3q^3+q^2$ & $0$ \\
	\hline
\end{tabular}
\end{center}

\vspace{5mm}
\begin{center} \label{table4}
\def\arraystretch{1.2}
\begin{tabular}{l|l|l|l} 
\multicolumn{4}{c} {\textbf{Table 4:} Some refined Kac functions for the quiver $\circlearrowleft\!\!\!\bullet \!\!\! \rightarrow \!\!\! \bullet \!\!\! \circlearrowright$ and $\bullet \!\!\! \rightarrow \!\!\! \bullet$} \\
	\hline
	$\,\lambda_*$ & $\lambda_*^\mathrm{v}$ & $A_\Gamma(\lambda_*, q)$ for $\circlearrowleft\!\!\!\bullet \!\!\! \rightarrow \!\!\! \bullet \!\!\! \circlearrowright$ $\text{\hspace{2 mm}}$
	& $A_\Gamma(\lambda_*, q)$ for $\bullet \!\!\! \rightarrow \!\!\! \bullet$ $\text{\hspace{2 mm}}$\\
	\hline
	$([0],[1])$ & ((0,1)) & $q$ & $1$ \\  
	$([0],[2])$ & ((0,0),(0,1)) & $q$ & $q^{-1}$ \\
	$([0],[1,1])$ & ((0,2)) & $0$ & $-q^{-1}$ \\
	$([1],[0])$ & ((1,0)) & $q$ & $1$ \\
     $([1],[1])$ & ((1,1)) & $q^2$ & $1$ \\
	$([1],[2])$ & ((1,0),(0,1)) & $q^2$ & $q^{-1}$ \\
	$([1],[1,1])$ & ((1,2)) & $q^3$ & $-q^{-1}$ \\
	$([2],[0])$ & ((0,0),(1,0)) & $q$ & $q^{-1}$ \\
     $([2],[1])$ & ((0,1),(1,0)) & $q^2$ & $q^{-1}$ \\
	$([2],[2])$ & ((0,0),(1,1)) & $q^3+q^2$ & $q^{-1}+q^{-2}$ \\
	$([2],[1,1])$ & ((0,2),(1,0)) & $q^3$ & $-q^{-2}$ \\
	$([1,1],[0])$ & ((2,0)) & $0$ & $-q^{-1}$ \\
     $([1,1],[1])$ & ((2,1)) & $q^3$ & $-q^{-1}$\\
	$([1,1],[2])$ & ((2,0),(0,1)) & $q^3$ & $-q^{-2}$\\
	$([1,1],[1,1])$ & ((2,2)) & $q^5+q^4$ & $-q^{-1}+q^{-2}$\\
	\hline
\end{tabular}
\end{center}

\vspace{0.2cm}
\textit{Email address}: \texttt{jiuzhao.hua@gmail.com}

\end{document}